\documentclass[11pt]{article}
\usepackage{fullpage,amssymb,url,amsmath,amsthm}
\usepackage{graphicx}

\def\calL{\mathcal{L}}
\def\eqdef{:=}
\def\hcross{{h^\times}}
\def\hcrossent#1#2{{\hcross}\left(#1:#2\right)}

\def\dmu{\mathrm{d}\mu}
\def\bbH{\mathbb{H}}
\def\dx{\mathrm{d}x}
\def\dy{\mathrm{d}y}
\def\KL{\mathrm{KL}}

\def\bbR{{\mathbb{R}}}
\def\calF{\mathcal{F}}
\def\calX{\mathcal{X}}
\def\calY{\mathcal{Y}}
\def\st{\ :\ }

\newtheorem{theorem}{Theorem}
\newtheorem{property}{Property}
\newtheorem{proposition}{Proposition}
\newtheorem{corollary}{Corollary}

\title{On the Kullback-Leibler divergence between location-scale densities}

\author{Frank Nielsen\\ Sony Computer Science Laboratories Inc, Japan}

\date{}

\begin{document}
\maketitle

\begin{abstract}
We show that the $f$-divergence between any two densities of potentially different location-scale families can be reduced to the calculation of the $f$-divergence between one standard density with another location-scale density.
It follows that the $f$-divergence between two scale densities depends only on the scale ratio.
We then report conditions on the standard distribution to get symmetric $f$-divergences:
First, we prove that all $f$-divergences between densities of a location family are symmetric whenever the standard density is even, and second, we illustrate a generic symmetric property with the calculation of the Kullback-Leibler divergence between scale Cauchy distributions.
Finally, we show that the  minimum $f$-divergence of any query density of a location-scale family  to
 another location-scale family is independent of the query location-scale parameters.
\end{abstract}


\section{Introduction}

Let $X\sim p$ be a random variable with cumulative distribution function $F_X$ and probability density $p_X(x)$ on the support $\calX$ (usually $\calX=\bbR$ or $\calX=\bbR_{++}$).
A location-scale random variable $Y=l+sX$ (for location parameter $l\in\calX$ and scale parameter $s>0$)  has distribution $F_Y(y)=F_X(\frac{x-l}{s})$ and density $p_Y(y)=p_X(\frac{x-l}{s})$. 
The location-scale  group~\cite{Murray-1993} $\bbH=\{(l,s) \ :\ l\in\bbR\times\bbR_{++}\}$ acts on the densities of a location-scale family~\cite{Murray-1993}:
The identity element is $i=(0,1)$, the group operation $e_1.e_2$ yields $e_1.e_2=(l_1+s_1l_2,s_1s_2)$ for $e_1=(l_1,s_1)$ and $e_2=(l_2,s_2)$, 
and the inverse element $e^{-1}$ is $e^{-1}=(-\frac{l}{s},\frac{1}{s})$ for $e=(l,s)$.

Consider two location-scale families~\cite{Murray-1993} sharing the same support $\calX$:
$$
\calF_1=  \left\{ p_{l_1,s_1}(x)= \frac{1}{s_1} p\left(\frac{x-l_1}{s_1}\right) \st (l_1,s_1)\in \bbH \right\},
$$
and 
$$
\calF_2=  \left\{ q_{l_2,s_2}(x)= \frac{1}{s_2} q\left(\frac{x-l_2}{s_2}\right) \st (l_2,s_2)\in \bbH \right\},
$$
where $p(x)=p_{0,1}(x)$ and $q(x)=q_{0,1}(x)$ denote the {\em standard densities} of $\calF_1$ and $\calF_2$, respectively (also called reduced distributions~\cite{KL-locationscale-2016}).

A {\em location family} is a subfamily of a location-scale family, with fixed scale $s_0$.
We denote by $p_l=p_{l,s_0}$ the density of a location family.
Similarly, a {\em scale family} is a subfamily of a location-scale family with prescribed location $l_0$.
We denote by $p_s=p_{l_0,s}$ the density of a scale family.

For example, $\calF_1$ can be the {\em Cauchy family}~\cite{KLCauchy-2019} with standard distribution $p(x)=\frac{1}{\pi (1+x^2)}$ and $\calF_2$ the {\em normal family} with
standard distribution $q(x)=\frac{1}{\sqrt{2\pi}}\exp(-\frac{x^2}{2})$, both families with support $\calX=(-\infty,\infty)$.
Then with our notations, $p_s=p_{0,s}$ is the Cauchy scale family and $q_l=q_{l,1}$ is the location normal family with unit variance.
Another example, is the family $\calF_1$ of half-normal distributions with the family $\calF_2$ of exponential distributions, both defined on the support 
$\calX=[0,\infty)$.

The {\em cross-entropy}~\cite{CT-2012,HcrossEF-2010} $\hcrossent{p_{l_1,s_1}}{q_{l_2,s_2}}$ between a density $p_{l_1,s_1}$ of $\calF_1$ and a density $q_{l_2,s_2}$ of $\calF_2$ is defined by
\begin{equation}\label{eq:htimes}
\hcrossent{p_{l_1,s_1}}{q_{l_2,s_2}} = -\int_{\calX} p_{l_1,s_1}(x)\log q_{l_2,s_2}(x) \dx.
\end{equation}

The {\em differential entropy}~\cite{CT-2012} $h$ is the self cross-entropy:
\begin{equation}\label{eq:h}
h(p_{l,s})=\hcrossent{p_{l,s}}{p_{l,s}}.
\end{equation}

The Kullback-Leibler (KL) divergence is the difference between the cross-entropy and the entropy:
\begin{eqnarray}
\KL(p_{l_1,s_1}:q_{l_2,s_2}) &=&  h^\times(p_{l_1,s_1}:q_{l_2,s_2}) - h(p_{l_1,s_1}) = 
\int_{\calX} p_{l_1,s_1}(x)\log \frac{p_{l_1,s_1}(x)}{q_{l_2,s_2}(x)} \dx \geq 0.
\end{eqnarray}

Note that the KL divergence between a standard Cauchy distribution and a standard Gaussian distribution is {\em infinite} since the integral diverges but the KL divergence between a standard Gaussian distribution and a standard Cauchy distribution is finite.
Thus the KL divergence between any two arbitrary location-scale families may potentially be infinite and may not admit a closed-form formula using the parameters $(l_1,s_1;l_2,s_2)$.

By making some changes of variable for $x$ in the cross-entropy integral on the right-hand-side of Eq.~\ref{eq:htimes}, we  establish 
the following four basic identities:

\begin{description}

\item[Left scale multiplication.]
\begin{equation}
\hcrossent{p_{l_1,\lambda_1 s_1}}{q_{l_2, s_2}} =   \hcrossent{p_{\frac{l_1}{\lambda_1},s_1}}{q_{\frac{l_2}{\lambda_1},\frac{s_2}{\lambda_1}}} +\log\lambda_1,\quad \forall \lambda_1\in\bbR_{++}.
\end{equation}

\begin{proof}
Make a change of variable in the integral with $y=\frac{x}{\lambda_1}$ for $\lambda_1>0$ (or $x=\lambda_1y$) and $\dx=\lambda_1 \dy$.
Then we have
\begin{eqnarray}
\hcrossent{p_{l_1,\lambda_1 s_1}}{q_{l_2, s_2}} &=&
-\int \frac{1}{\lambda_1 s_1} p\left(\frac{x-l_1}{\lambda_1 s_1}\right)  \log \frac{1}{s_2} q\left(\frac{x-l_2}{s_2}\right) \dx,\\
&=&  -\int \frac{1}{s_1} p\left(\frac{y-\frac{l_1}{\lambda_1}}{s_1}\right)  \log \frac{\lambda_1}{s_2}\frac{1}{\lambda_1} 
q\left(\frac{y-\frac{l_2}{\lambda_1}}{\frac{s_2}{\lambda_1}}\right) \dy,\\
&=& -\int p_{\frac{l_1}{\lambda_1},s_1}(y) \log \frac{1}{s_2}\frac{\lambda_1}{\lambda_1} 
q\left(\frac{y-\frac{l_2}{\lambda_1}}{\frac{s_2}{\lambda_1}}\right) \dy + \log\lambda_1   \int p_{\frac{l_1}{\lambda_1},s_1}(y)\dy,\\
&=& \hcrossent{p_{\frac{l_1}{\lambda_1},s_1}}{q_{\frac{l_2}{\lambda_1},\frac{s_2}{\lambda_1}}} +\log\lambda_1.
\end{eqnarray}
\end{proof}
From now on, we skip the other substitution proofs that are similar and state the identities:

\item[Left location translation.]
\begin{equation}
\hcrossent{p_{l_1+\alpha_1,s_1}}{q_{l_2,s_2}} =  \hcrossent{p_{l_1,s_1}}{q_{l_2-\alpha_1,s_2}},\quad \forall \alpha_1\in\bbR.
\end{equation}

\item[Right scale multiplication.]
\begin{equation}
\hcrossent{p_{l_1,s_1}}{q_{l_2,\lambda_2 s_2}} = \hcrossent{p_{\frac{l_1}{\lambda_2},\frac{s_1}{\lambda_2}}}{q_{\frac{l_2}{\lambda_2},s_2}} + \log\lambda_2,
\quad \forall \lambda_2\in\bbR_{++}.
\end{equation}

\item[Right location translation.]
\begin{equation}
\hcrossent{p_{l_1,s_1}}{q_{l_2+\alpha_2,s_2}} = \hcrossent{p_{l_1-\alpha_2,s_1}}{q_{l_2,s_2}}, \quad \forall \alpha_2\in\bbR.
\end{equation}

\end{description}

Furthermore, we get the following double-sided  scale multiplication identity by a change of variable (can also be obtained by applying the left scale multiplication with parameter $\lambda_1=\sqrt{\lambda}$ and then the right scale multiplication with parameter $\lambda_2=\sqrt{\lambda}$ :
\begin{equation}
\hcrossent{p_{l_1,\lambda s_1}}{q_{l_2,\lambda s_2}} = 
\hcrossent{p_{\frac{l_1}{\lambda},s_1}}{q_{\frac{l_2}{\lambda},s_2}} + \log \lambda,\quad \forall \lambda>0,
\end{equation}
and the generic cross-entropy rule by translations:
\begin{equation}
\hcrossent{p_{l_1+\alpha,s_1}}{q_{l_2+\beta,s_2}} = \hcrossent{p_{l_1,s_1}}{q_{l_2+\beta-\alpha,s_2}} = 
\hcrossent{p_{l_1+\alpha-\beta,s_1}}{q_{l_2,s_2}},\quad \forall \alpha,\beta\in\bbR.
\end{equation}

By using these ``parameter rewriting'' rules, we get the following properties:

\begin{property}[Location-scale entropy]\label{prop:ent}
We have 
\begin{equation}
h(p_{l,s})=h(p)+\log s.
\end{equation}
\end{property}
That is, the differential entropy of a density $p_{l,s}$ of a location-scale family is {\em independent of the location} and can be calculated from the entropy of the standard density $p$.

\begin{proof}
We have $h(p_{l,s})=\hcrossent{p_{l,s}}{p_{l,s}}=\hcrossent{p_{0,s}}{p_{0,s}}$ (using either the left/right translation rule) and
$\hcrossent{p_{0,s}}{p_{0,s}}=\hcrossent{p_{0,1}}{p_{0,1}}+\log s=h(p)+\log s$ (using either left/right multiplication rule).
\end{proof}

\begin{property}[Location-scale cross-entropy]\label{prop:cross}
We have 
\begin{eqnarray}
\hcrossent{p_{l_1,s_1}}{q_{l_2,s_2}} &=& \hcrossent{p_{\frac{l_1-l_2}{s_2},\frac{s_1}{s_2}}}{q}+\log s_2,\\
&=&
\hcrossent{p}{q_{\frac{l_2-l_1}{s_1},\frac{s_2}{s_1}}}+\log s_1.
\end{eqnarray}
\end{property}

That is, the cross-entropy between two location-scale densities can be reduced to the calculation of the cross-entropy between a standard density and a density of the other location-scale family.

\begin{proof}
We have $\hcrossent{p_{l_1,s_1}}{q_{l_2,s_2}}=\hcrossent{p_{\frac{l_1}{s_2},\frac{s_1}{s_2}}}{q_{\frac{l_2}{s_2},1}}+\log s_2$ (right multiplication rule) and $\hcrossent{p_{\frac{l_1}{s_2},\frac{s_1}{s_2}}}{q_{\frac{l_2}{s_2},1}}=\hcrossent{p_{\frac{l_1}{s_2}-\frac{l_2}{s_2},\frac{s_1}{s_2}}}{q_{0,1}}=\hcrossent{p_{\frac{l_1}{s_2}-\frac{l_2}{s_2},\frac{s_1}{s_2}}}{q}$ (right translation rule).
\end{proof}

\begin{property}[Location-scale Kullback-Leibler divergence]\label{prop:kl}
We have 
\begin{eqnarray}
\KL(p_{l_1,s_1}:q_{l_2,s_2})
&=&
\hcrossent{p}{q_{\frac{l_2-l_1}{s_1},\frac{s_2}{s_1}}}-h(p)=\KL\left(p:q_{\frac{l_2-l_1}{s_1},\frac{s_2}{s_1}}\right),\\
 &=& \hcrossent{p_{\frac{l_1-l_2}{s_1},\frac{s_1}{s_2}}}{q}-h(p)+\log \frac{s_2}{s_1}  = \KL(p_{\frac{l_1-l_2}{s_2},\frac{s_1}{s_2}}:q).
\end{eqnarray}
\end{property}

\begin{proof}
We have $\KL(p_{l_1,s_1}:q_{l_2,s_2})=\hcrossent{p_{l_1,s_1}}{q_{l_2,s_2}}-h(p_{l_1,s_1})$.
Then we apply Property~\ref{prop:cross}  
$\hcrossent{p_{l_1,s_1}}{q_{l_2,s_2}}=\hcrossent{p}{q_{\frac{l_2-l_1}{s_1},\frac{s_2}{s_1}}}+\log s_1$
and Propery~\ref{prop:ent}   $h(p_{l_1,s_1})=h(p)+\log s_1$ to get the result (the terms $\log s_1$  cancel out).
\end{proof}

Similarly, we have the following basic identities for the Kullback-Leibler divergence between any two location-scale densities:

\begin{eqnarray}
\KL(p_{l_1+l,s_1}:q_{l_2,\lambda s_2}) &=& \KL(p_{l_1,s_1}:q_{l_2-l,\lambda s_2}),\\
\KL(p_{l_1,\lambda  s_1}:q_{l_2,s_2}) &=& \KL(p_{\frac{l_1}{\lambda},{s_1}}:q_{\frac{l_2}{\lambda},\frac{s_2}{\lambda}}),\\
\KL(p_{l_1,s_1}:q_{l_2+l,s_2}) &=& \KL(p_{l_1-l,s_1}:q_{l_2,s_2}),\\
\KL(p_{l_1,s_1}:q_{l_2,\lambda s_2}) &=& \KL(p_{\frac{l_1}{\lambda},\frac{s_1}{\lambda}}:q_{\frac{l_2}{\lambda},s_2}).
\end{eqnarray}

We state the following theorem:

\begin{theorem}\label{thm:scaleinvariant}
The Kullback-Leibler divergence between two densities belonging to the same scale family is scale invariant:
$\KL(p_{\lambda s_1}:p_{\lambda s_2}) = \KL(p_{s_1}:p_{s_2})$ for any $\lambda>0$.
\end{theorem}

\begin{proof}
We have $\KL(p_{\lambda s_1}:p_{\lambda s_2})=\KL(p:p_{\frac{\lambda s_2}{\lambda s_1}})=\KL(p:p_{\frac{s_2}{s_1}})=\KL(p_{s_1}:p_{s_2})$.
\end{proof}

We can define a {\em scalar divergence} $D(s_1:s_2)\eqdef \KL(p_{s_1}:p_{s_2})$ that is scale-invariant: $D(\lambda s_1:\lambda s_2)=D(s_1:s_2)$ for any $\lambda>0$.
Another common example of scalar divergence is the Itakura-Saito divergence which belongs to the class of Bregman divergences~\cite{NMF-2009}.

The result presented for the KL divergence holds in the more general setting of Csisz\'ar's $f$-divergences~\cite{Csiszar-1963,TaylorFdiv-2013}: 
\begin{equation}
I_f(p:q)=\int_{\calX} p(x) f\left(\frac{q(x)}{p(x)}\right)\dx,
\end{equation}
for a positive convex function $f$, strictly convex at $1$, with $f(1)=0$.
The KL divergence is a $f$-divergence for the generator $f(u)=-\log u$.

\begin{theorem}
The $f$-divergence between two location-scale densities $p_{l_1,s_1}$ and $q_{l_2,s_2}$ can be reduced to the calculation of the $f$-divergence between one standard density with another location-scale density:
\begin{equation}
I_f(p_{l_1,s_1}:q_{l_2,s_2}) = I_f\left(p:q_{\frac{l_2-l_1}{s_1},\frac{s_2}{s_1}}\right) 
= I_f\left(p_{\frac{l_1-l_2}{s_2},\frac{s_1}{s_2}}:q\right).
\end{equation}
\end{theorem}

\begin{proof}
The proofs follow from changes of the variable $x$ in the integral:
Consider $y=\frac{x-l_1}{s_1}$ with $\dx=s_1\dy$, $x=s_1y+l_1$ and $\frac{x-l_2}{s_2}=\frac{s_1y+l_1-l_2}{s_2}=\frac{y-\frac{l_2-l_1}{s_1}}{\frac{s_2}{s_1}}$:
\begin{eqnarray}
I_f(p_{l_1,s_1}:q_{l_2,s_2}) &:=& \int_{\calX} p_{l_1,s_1}(x) f\left(\frac{q_{l_2,s_2}(x)}{p_{l_1,s_1}(x)}  \right)\dx,\\
&=& \int_{\calY} \frac{1}{s_1} p(y) f\left( \frac{\frac{1}{s_2} q\left(\frac{y-\frac{l_2-l_1}{s_1}}{\frac{s_2}{s_1}}\right)}{\frac{1}{s_1}p(y)}
\right) s_1\dy,\\
&=& \int p(y) f\left(\frac{q_{\frac{l_2-l_1}{s_1},\frac{s_2}{s_1}}(y)}{p(y)}\right) \dy,\\
&=&  I_f\left(p:q_{\frac{l_2-l_1}{s_1},\frac{s_2}{s_1}}\right).
\end{eqnarray}
The proof for  $I_f(p_{l_1,s_1}:q_{l_2,s_2})=I_f(p_{\frac{l_1-l_2}{s_2},\frac{s_1}{s_2}}:q)$ is similar, or one can use the adjoint generator 
$f^*(u)=uf(\frac{1}{u})$ which yields the reverse $f$-divergence: $I_{f^*}(p:q)=I_f(q:p)$.
\end{proof}

Note that $f$-divergences are invariant under any diffeomorphism~\cite{IG-2016} $y=t(x)$ of the sample space $\calX$.
The $f$-divergences are called invariant divergences in information geometry~\cite{IG-2016}.
In particular, this invariance property includes the diffeomorphism defined by the group action of the location-scale group.

Thus  the $f$-divergences between scale densities amount to a scale-invariant scalar distance:
\begin{eqnarray}
D_f(s_1:s_2) := I_f(p_{s_1}:q_{s_2}) &=&   I_f\left(p:q_{\frac{s_2}{s_1}}\right) =: D_f\left(1:\frac{s_2}{s_1}\right),\\
 &=& I_f\left(p_{\frac{s_1}{s_2}}:q\right)=: 
D_f\left(\frac{s_1}{s_2}:1\right).
\end{eqnarray}

\section{The KL divergence between Cauchy location/scale distributions}

In this section, we consider a working example for the scale Cauchy family.
Surprisingly, the formula has not been widely reported in the literature (an erratum\footnote{see \url{https://sites.google.com/site/geotzag/publications}} corrects the formula given in~\cite{cauchy-2004}). 
Note that the Cauchy scale family can also be interpreted as a $q$-Gaussian family for $q=2$~\cite{IG-2016} and
a $\alpha$-stable family~\cite{cauchy-2004} for $\alpha=1$.

Consider the cross-entropy between two Cauchy distributions $p_1$ and $p_2$.
Using Property~\ref{prop:kl}, we can assume without loss of generality that the distribution $p_2$ is the standard Cauchy distribution $p$, and focus on calculating the following cross-entropy:
\begin{equation}
\hcrossent{p_{l,s}}{p} = -\int_{-\infty}^{\infty} p_{l,s}(x)\log p(x) \dx,
\end{equation}
with location $l=\frac{l_1-l_2}{s_1}$ and scale $s=\frac{s_1}{s_2}$, where
\begin{equation} 
p(x)= \frac{1}{\pi (1+x^2)},\quad p_{l,s}(x) = \frac{s}{\pi ( s^2+(x-l)^2)}.
\end{equation}

The scale Cauchy distributions form a subfamily with $l=0$.
We shall use the following result on definite integrals  (listed under the logarithmic forms of definite integrals in many handbooks of formulas and tables):\footnote{Also listed online at \url{https://en.wikipedia.org/wiki/List_of_definite_integrals}}
\begin{equation} 
A(a,b)=\int_{-\infty}^\infty \frac{\log (a^2+x^2)}{b^2+x^2}\dx = \frac{2\pi}{b}\log(a+b),\quad a,b>0.
\end{equation}

We get the cross-entropy between two scale Cauchy distributions $p_{s_1}$ and $p_{s_2}$ as follows:
\begin{eqnarray}
\hcrossent{p_{s_1}}{p_{s_2}} &=& \hcrossent{p_{s}}{p} +   \log s_2 , \\
&=&  \frac{s}{\pi} \int \frac{1}{s^2+x^2} \log  (1+x^2) \dx + \log \pi +\log s_2,\\
&=& \log \pi s_2 + \frac{s}{\pi} I(1,s),\\
&=& \log \pi\frac{(s_1+s_2)^2}{s_2}.
\end{eqnarray}

The differential entropy is obtained for $s_1=s_2=s$:
\begin{equation} 
h(p_{s})=\hcrossent{p_{s}}{p_{s}}=\log 4\pi s,
\end{equation} 
in accordance with~\cite{h-handbook-2013} (p. 68).
Thus the Kullback-Leibler between two scale Cauchy distributions is:

\begin{eqnarray}
\KL(p_{s_1}:p_{s_2}) = \hcrossent{p_{s_1}}{p_{s_2}}-h(p_{s_1}) &=& 2\log \left(\frac{s_1+s_2}{2\sqrt{s_1s_2}}\right),\label{eq:klcauchy}\\
&=&  2\log \left(\frac{1+\frac{s_2}{s_1}}{2\sqrt\frac{s_2}{s_1}}\right)
 = 2\log \left(\frac{1+\frac{s_1}{s_2}}{2\sqrt\frac{s_1}{s_2}}\right).
\end{eqnarray}
Notice that $A(s_1,s_2)=\frac{s_1+s_2}{2}$ is the arithmetic mean of the scales, and
$G(s_1,s_2)=\sqrt{s_1s_2}$ is the geometric mean of the scales.
Thus the KL divergence can be rewritten as $\KL(p_{s1}:p_{s_2})=2\log \frac{A(s_1,s_2)}{G(s_1,s_2)}$.
Since we have the arithmetic-geometric inequality $A\geq G$ (and $\frac{A}{G}\geq 1$), it follows that $\KL(p_{s_1}:p_{s_2})\geq 0$.

Let us notice that the KL divergence between two Cauchy scale distributions is symmetric: $\KL(p_{s_1}:p_{s_2})=\KL(p_{s_2}:p_{s_1})$.
For exponential families~\cite{EF-2009}, the KL divergence is provably symmetric only for the location (multivariate/elliptical) Gaussian family since the KL divergence amount to a Bregman divergence, and the only symmetric Bregman divergences are the squared Mahalanobis distances~\cite{BVD-2010}.
Not all scale families are symmetric:
For example, the Rayleigh distributions form a scale family (and also an exponential family~\cite{EF-2009}) but the KL divergence between two Rayleigh distributions amount to an Itakura-Saito divergence~\cite{EF-2009} that is asymmetric.

\begin{proposition}[KLD between scale Cauchy densities]
The differential entropy, cross-entropy and Kullback-Leibler divergence between two scale Cauchy densities $p_{s_1}$ and $p_{s_2}$ are:
\begin{eqnarray*}
h(p_{s}) &=& \log 4\pi s,\\
\hcrossent{p_{s_1}}{p_{s_2}} &=& \log \pi\frac{(s_1+s_2)^2}{s_2},\\
\KL(p_{s_1}:p_{s_2}) &=&  2\log \left(\frac{s_1+s_2}{2\sqrt{s_1s_2}}\right).
\end{eqnarray*}
\end{proposition}

\begin{corollary}
The Kullback-Leibler divergence between two Cauchy scale distributions is scale invariant.
\end{corollary}

\begin{proof}
Theorem~\ref{thm:scaleinvariant} already proves this property for any scale family including the Cauchy scale family.
However, here we shall directly use the property of homogeneous means.
Since for all $\lambda>0$, we have $A(\lambda s_1,\lambda s_2)=\lambda A(s_1,s_2)$ and $G(\lambda s_1,\lambda s_2)=\lambda G(s_1,s_2)$, it follows that $\frac{A(\lambda s_1,\lambda s_2)}{G(\lambda s_1,\lambda s_2)}=\frac{A(s_1,s_2)}{G(s_1,s_2)}$, and we have $\KL(p_{\lambda s_1}:p_{\lambda s_2})=\KL(p_{s_1}:p_{s_2})$.
\end{proof}

Let us mention the generic formula~\cite{KLCauchy-2019} for the Kullback-Leibler divergence between Cauchy location-scale density $p_{l_1,s_1}$ and $p_{l_2,s_2}$ is
\begin{equation}
\KL(p_{l_1,s_1}:p_{l_2,s_2})=\log\frac{(s_1+s_2)^2+(l_1-l_2)^2}{4s_1s_2}.
\end{equation}

Notice that the $f$-divergence is invariant by a diffeormorphism of the sample space:
For example, consider the log-normal scale family of location $l$ and scale $s$, and consider the mapping $x=\log z$.
Then we obtain a normal distribution of location $l$ and scale $s$.

\section{Conditions on the standard density for symmetric KLDs}

\subsection{The $f$-divergences between densities of a location family}
We first consider the case of location families.

\begin{theorem}
Let $\calL=\{p(x-l)\ :\ l\in\bbR\}$ denote a location family with even standard density $p(-x)=p(x)$ on the support $\bbR$.
Then all $f$-divergences between two densities $p_{l_1}$ and $p_{l_2}$ of $\calL$ are symmetric: 
$I_f[p_{l_1}:p_{l_2}]=I_f[p_{l_2}:p_{l_1}]$.
\end{theorem}

\begin{proof}
Consider the change of variable $l_1-x=y-l_2$ (so that $x-l_2=l_1-y$) with $\dx=-\dy$ and let us use the property that $p(z-l_1)=p(l_1-z)$ since $p(z)$ is an even standard density.
We have:

\begin{eqnarray}
I_f[p_{l_1}:p_{l_2}] &:=& \int_{-\infty}^{+\infty} p(x-l_1) f\left(\frac{p(x-l_2)}{p(x-l_1)}\right) \dx,\\
&=& \int_{+\infty}^{-\infty} p(l_1-x) f\left(\frac{p(x-l_2)}{p(l_1-x)}\right)  (-\dy),\\
&=&  \int_{-\infty}^{+\infty} p(y-l_2) f\left(\frac{p(x-l_2)}{p(y-l_2)}\right)  \dy,\\
&=& \int_{-\infty}^{+\infty} p(y-l_2) f\left(\frac{p(l_1-y)}{p(y-l_2)}\right)  \dy,\\
&=& \int_{-\infty}^{+\infty} p(y-l_2) f\left(\frac{p(y-l_1)}{p(y-l_2)}\right)  \dy,\\
&=:& I_f[p_{l_2}:p_{l_1}].
\end{eqnarray}
\end{proof}

For example, the $f$-divergences between location Cauchy densities are symmetric since $p(x)=p(-x)$ for the standard Cauchy density.

\subsection{A general symmetric condition}

Let us study when the KL divergence between location-scale families is symmetric by characterizing the standard distribution: 
$\KL(p_{l_1,s_1}:p_{l_2,s_2}) = \KL(p_{l_2,s_1}:p_{l_1,s_1})$.
Since $\KL(p_{l_1,s_1}:p_{l_2,s_2})=\KL(p:p_{l,s})$ (with $s=\frac{s_2}{s_1}$ and $l=\frac{l_2-l_1}{s_1}$), we consider the case where

\begin{equation}
\KL(p:p_{l,s}) = \KL(p_{l,s}:p) = \KL\left(p:p_{\frac{1}{s},-\frac{l}{s}}\right).
\end{equation}

The equality $\KL(p:p_{l,s})=\KL(p:p_{\frac{1}{s},-\frac{l}{s}})$ yields the following equivalent condition:

\begin{equation}
\int_\calX p(x)\log\frac{p_{\frac{1}{s},-\frac{l}{s}}(x)}{p_{l,s}(x)}\dx =0,\quad \forall l\in\bbR,s>0.
\end{equation}

Assume a location family (i.e., $s=1$), then we have the condition
\begin{equation}
\int_\calX p(x)\log\frac{p_{-l}(x)}{p_{l}(x)}\dx =0,\quad \forall l\in\bbR,s>0.
\end{equation}
Since $p_{-l}(x)=p(x+l)$ and $p_{l}(x)=p(x-l)$, we end up with the condition
\begin{equation}\label{eq:klsymloc}
\int_\calX p(x)\log\frac{p(x+l)}{p(x-l)}\dx =0,\quad \forall l\in\bbR.
\end{equation}

For example, the location normal distribution has symmetric KL divergence because it satisfies Eq.~\ref{eq:klsymloc}.
Indeed, for normal location distributions, we have
$\int_\calX p(x)\log\frac{p_{-l}(x)}{p_{l}(x)}\dx=2l\int_\calX 2x p(x)=2l E[x]=0$ since $p(x)$ for the standard Gaussian density is an even function.

Consider now the scale family (with $l=0$), then we find the following condition

\begin{equation}\label{eq:klsymscale}
\int_\calX p(x)\log\frac{p(\frac{x}{s})}{p(sx)}\dx =2\log s,\quad \forall s\in\bbR_{++}.
\end{equation}

For example, the Cauchy scale distribution has symmetric KL divergence because the Cauchy standard distribution satisfies Eq.~\ref{eq:klsymscale}:

\begin{eqnarray}
\int_\calX p(x)\log\frac{p(\frac{x}{s})}{p(sx)}\dx &=& \frac{1}{\pi} \log \frac{\pi(1+s^2x^2)}{\pi 1+\frac{x^2}{s^2}},\\
&=& \frac{1}{\pi} \left( A(s,1)-A\left(\frac{1}{s},1\right)\right),\\
&=& 2 \log \frac{1+s}{1+\frac{1}{s}} = 2\log s.
\end{eqnarray}

Similarly, the $f$-divergence between two densities $p$ and $q$ is symmetric if and only if: 
\begin{equation}
\int_{\calX} \left( p(x)f\left(\frac{q(x)}{p(x)}\right) -q(x) f\left(\frac{p(x)}{q(x)}\right)\right) \dx=0.
\end{equation}

\section{Kullback-Leibler minimizations between location-scale families}
Consider the density manifold~\cite{DensityMfd-1988} $M$ (Fr\'echet manifold), and two densities $p$ and $q$ of $M$.
We can generate the location-scale families/submanifolds 
$P=\{\frac{1}{s}p\left(\frac{x-l}{s}\right) \ :\ (l,s)\in \bbH\}$ and $Q=\{\frac{1}{s}q\left(\frac{x-l}{s}\right) \ :\ (l,s)\in \bbH\}$, where $\bbH=\bbR\times \bbR_{++}$ is the open half-space of 2D location-scale parameters.

Consider the following Kullback-Leibler minimization problem:
\begin{eqnarray}
\KL(p_{l_1,s_1}:Q) := \min_{(l_2,s_2) \in \bbH} &&\KL(p_{l_1,s_1}:q_{l_2,s_2})\\
\equiv \min_{(l_2,s_2) \in \bbH} &&\KL(p:q_{\frac{l_2-l_1}{s_1},\frac{s_2}{s_1}})\\
\equiv \min_{(l,s) \in \bbH} &&\KL(p:q_{l,s}) := \KL(p:Q),\\
\end{eqnarray}
with $l=\frac{l_2-l_1}{s_1}$ and $s=\frac{s_2}{s_1}$.
Once the best parameters $(l^*,s^*)$ have been calculated for a query density $p_{l_1,s_1}$, we get the minimizer on the other location-scale family as $l_2^*=s_1l^*+l_1$
and $s_2^*=s^*s_1$.

We have $\KL(p:q_{l,s})=h^\times(p:q_{l,s})-h(p)$, and therefore $\min_{(l,s) \in \bbH}\KL(p:q_{l,s})$ 
amount to  $\max_{(l,s) \in \bbH}  \int p(x)\log q_{l,s}(x)\dmu(x)$.

\begin{theorem}
The minimum KL divergence $\KL(p_{l_1,s_1}:q_{l_1^*,s_1^*})$ induced by the right-sided KL minimization of $p_{l_1,s_1}$ with $Q$  is independent of the location-scale query parameter $(l_1,s_1)$.
Similarly, the KL divergence $\KL(p_{l_2^*,s_2^*}:q_{l_2,s_2})$ induced by the left-sided KL minimization of $q_{l_2,s_2}$ with $P$  is independent of the location-scale query parameter $(l_2,s_2)$.
\end{theorem}

Notice that in general $\KL(p_{l_1,s_1}:q_{l_1^*,s_1^*})\not = \KL(p_{l_2^*,s_2^*}:q_{l_2,s_2})$.
The theorem is a statement of the property mentioned without proof in~\cite{KL-locationscale-2016}.

The proof extends easily to $f$-divergences as follows:
\begin{theorem}
The $f$-divergence $I_f(p_{l_1,s_1}:q_{l_1^*,s_1^*})$ induced by the right-sided $f$-divergence minimization of $p_{l_1,s_1}$ with $Q$  is independent of $(l_1,s_1)$.
Similarly, the $f$-divergence $I_f(p_{l_2^*,s_2^*}:q_{l_2,s_2})$ induced by the left-sided  $f$-divergence minimization of $q_{l_2,s_2}$ with $P$ is independent of $(l_2,s_2)$.
\end{theorem}

\begin{proof}
Without loss of generality, consider the left-sided $f$-divergence minimization problem (right-sided density query).
We have
\begin{equation}
I_f(P:q_{l_2,s_2}) := \min_{(l_1,s_1)\in \bbH} I_f(p_{l_1,s_1}:q_{l_2,s_2}) =
 \min_{(l_1,s_1)\in \bbH} I_f\left(p:q_{\frac{l_2-l_1}{s_1},\frac{s_2}{s_1}}\right).
\end{equation}
Let $l=\frac{l_2-l_1}{s_1}$ and $s=\frac{s_2}{s_1}$. Then the minimization problem becomes:
\begin{equation}
\min_{(l_1,s_1)\in \bbH} I_f(p_{l_1,s_1}:q_{l_2,s_2}) = \min_{(l,s)\in \bbH} I_f(p:q_{l,s}) := I_f(p:Q).
\end{equation}
Once the optimal parameter $l^*$ and $s^*$ have been calculated, we recover the density $p_{l_1^*,s_1^*}\in P$ that minimizes $I_f(p_{l_1,s_1}:q_{l_2,s_2})$ as $p_{l_1^*,s_1^*}$ with
\begin{eqnarray}
s_1^* &=& \frac{s_2}{s^*},\\
l_1^* &=& l_2-l^*s_1^*.
\end{eqnarray}
 
\end{proof}

Let us remark that these $f$-divergence minimization problems between a query density and a location-scale family can be interpreted as information projections~\cite{infoproj-2018} of a query density onto a location-scale manifold.

As a corollary, observe that 
\begin{eqnarray}
I_f(P:Q) &:=& \min_{(l_1,s_1)\in\bbH, (l_2,s_2)\in\bbH} I_f(p_{l_1,s_1}:q_{l_2,s_2}),\\
&=& \min\{ \min_{(l_2,s_2)\in\bbH} I_f(p:q_{l_2,s_2}), \min_{(l_1,s_1)\in\bbH} I_f(p_{l_1,s_1}:q) \}\\
&=& = \min\{I_f(P:q)= I_f(p:Q)\}.
\end{eqnarray}

Let us rework the example originally reported in~\cite{KL-locationscale-2016}:
Consider $p(x)=\sqrt{\frac{2}{\pi}}\exp(-\frac{x^2}{2})$ and $q(x)=\exp(-x)$ be the standard density of the half-normal distribution  and the standard density of the exponential distribution defined over the support $\calX=[0,\infty)$, respectively. We consider the scale families $P=\{p_{s_1}(x)=\frac{1}{s_1}p(\frac{x}{s_1}) \st s_1>0\}$ and $Q=\{q_{s_2}(x)=\frac{1}{s_2}q(\frac{x}{s_2}) \st s_2>0\}$.
Using symbolic computing detailed in Appendix~\ref{sec:maxima}, we find that
\begin{equation}\label{eq:hne}
\KL(p_{s_1}:q_{s_2}) = \frac{1}{2}\left(2\log\frac{s_2}{s_1} + \log\frac{2}{\pi} -1 \right) + \sqrt{\frac{2}{\pi}} \frac{s_1}{s_2}.
\end{equation}

Let $r=\frac{s_1}{s_2}$. Then $\KL(p_{s_1}:q_{s_2})=\sqrt{\frac{2}{\pi}}r-\log r+\log \sqrt{\frac{2}{\pi}}-\frac{1}{2}$.
That is, the KL between the scale families depends only on the scale ratio as proved earlier.

We KL divergence is minimized when $-\frac{1}{r}+\sqrt{\frac{2}{\pi}}=0$. That is, when $r=\sqrt{\frac{\pi}{2}}$.
We find that $\KL(p_{s_1}:Q)=\frac{1}{2}+\log \frac{2}{\pi}$ is independent of $s_1$, as expected.

\section{Concluding remarks}

The canonical structure of the densities of the location-scale families make it possible to get various identities for the cross-entropy, the differential entropy, and the Kullback-Leibler divergence, by making change of variables in the corresponding integrals. 
In particular, the Kullback-Leibler divergence (or more generally any $f$-divergence) between location-scale densities can be reduced to the calculation of the 
Kullback-Leibler divergence between one standard density with another transformed location-scale density.
It follows that the Kullback-Leibler divergence between scale densities depends only on the scale ratio.
We illustrated our approach by computing the Kullback-Leibler divergence between scale Cauchy distributions which is symmetric.
More generally, we reported a condition on the standard density of a location-scale family which yields symmetric Kullback-Leibler divergences. 
We proved that all $f$-divergences between two densities of a location family are symmetric provided that the standard density is an even function.
Finally, we proved that the minimum $f$-divergence between a query density of a location-scale family with any member of another location-scale family does not depend on the query location-scale parameters.
To conclude, let us mention that we can derive similarly identities for information-theoretic measures from change of variables in integrals for location-dispersion families~\cite{LD-2007}.

\appendix

\section{Symbolic calculation using {\sc Maxima}}\label{sec:maxima}

We use the computer algebra system {\sc Maxima}\footnote{http://maxima.sourceforge.net/} to calculate Eq.~\ref{eq:hne}:

\begin{verbatim}
pe(x) := sqrt(2/%pi)*exp(-x*x/(2.0));
qe(x) := exp(-x);
p(x,s1) := (1/s1)*pe(x/s1);
q(x,s2) := (1/s2)*qe(x/s2);
assume(s1>0);
assume(s2>0);
integrate(p(x,s1)*log(p(x,s1)/q(x,s2)),x,0,inf);
ratsimp(%);
\end{verbatim}


\begin{thebibliography}{10}

\bibitem{IG-2016}
S.~Amari.
\newblock {\em Information Geometry and Its Applications}.
\newblock Applied Mathematical Sciences. Springer Japan, 2016.

\bibitem{BVD-2010}
Jean-Daniel Boissonnat, Frank Nielsen, and Richard Nock.
\newblock Bregman {V}oronoi diagrams.
\newblock {\em Discrete \& Computational Geometry}, 44(2):281--307, 2010.

\bibitem{KLCauchy-2019}
Fr\'ed\'eric Chyzak and Frank Nielsen.
\newblock A closed-form formula for the {K}ullback-{L}eibler divergence between
  {C}auchy distributions.
\newblock {\em arXiv preprint arXiv:1905.10965}, 2019.

\bibitem{CT-2012}
Thomas~M Cover and Joy~A Thomas.
\newblock {\em Elements of information theory}.
\newblock John Wiley \& Sons, 2012.

\bibitem{Csiszar-1963}
Imre Csisz\'ar.
\newblock Eine informationstheoretische ungleichung und ihre anwendung auf den
  beweis der ergodizitat von markoffschen ketten.
\newblock {\em Magyar. Tud. Akad. Mat. Kutat\'o Int. K\"ozl}, 8:85--108, 1963.

\bibitem{NMF-2009}
C{\'e}dric F{\'e}votte, Nancy Bertin, and Jean-Louis Durrieu.
\newblock Nonnegative matrix factorization with the {I}takura-{S}aito
  divergence: With application to music analysis.
\newblock {\em Neural computation}, 21(3):793--830, 2009.

\bibitem{DensityMfd-1988}
John~D Lafferty.
\newblock The density manifold and configuration space quantization.
\newblock {\em Transactions of the American Mathematical Society},
  305(2):699--741, 1988.

\bibitem{h-handbook-2013}
Joseph~Victor Michalowicz, Jonathan~M Nichols, and Frank Bucholtz.
\newblock {\em Handbook of differential entropy}.
\newblock CRC Press, 2013.

\bibitem{Murray-1993}
M.~K. Murray and J.~W. Rice.
\newblock {\em Differential Geometry and Statistics}, volume~48.
\newblock CRC Press, 1993.

\bibitem{infoproj-2018}
Frank Nielsen.
\newblock What is an information projection?
\newblock {\em Notices of the AMS}, 65(3):321--324.

\bibitem{EF-2009}
Frank Nielsen and Vincent Garcia.
\newblock Statistical exponential families: A digest with flash cards.
\newblock {\em arXiv preprint arXiv:0911.4863}, 2009.

\bibitem{HcrossEF-2010}
Frank Nielsen and Richard Nock.
\newblock Entropies and cross-entropies of exponential families.
\newblock In {\em Image Processing (ICIP), 2010 17th IEEE International
  Conference on}, pages 3621--3624. IEEE, 2010.

\bibitem{TaylorFdiv-2013}
Frank Nielsen and Richard Nock.
\newblock On the chi square and higher-order chi distances for approximating
  $f$-divergences.
\newblock {\em IEEE Signal Processing Letters}, 21(1):10--13, 2013.

\bibitem{LD-2007}
Toshio Ohnishi and Takemi Yanagimoto.
\newblock Conjugate location-dispersion families.
\newblock {\em Journal of the Japan Statistical Society}, 37(2):307--325, 2007.

\bibitem{cauchy-2004}
George Tzagkarakis and Panagiotis Tsakalides.
\newblock A statistical approach to texture image retrieval via alpha-stable
  modeling of wavelet decompositions.
\newblock In {\em International Workshop on Image Analysis for Multimedia
  Interactive Services}, pages 21--23, 2004.

\bibitem{KL-locationscale-2016}
Cristiano {Villa}.
\newblock {A Property of the {K}ullback-{L}eibler Divergence for Location-scale
  Models}.
\newblock {\em ArXiv e-prints}, April 2016.

\end{thebibliography}
\end{document}